\documentclass{amsart}
\usepackage[utf8]{inputenc}
\usepackage[margin=1in, marginparwidth=2cm]{geometry}
\usepackage{amsmath,amssymb,amsthm}
\usepackage{amsfonts}
\usepackage{etoolbox}
\usepackage{mathdots}
\usepackage{mathtools}
\usepackage{tikz-cd}
\usepackage{hyperref}
\usepackage[capitalize]{cleveref}

\numberwithin{equation}{section}
\def\equationautorefname#1#2\null{%
	(#2\null)%
}

\newcommand{\ZZ}{\mathbb Z}
\newcommand{\NN}{\mathbb N}
\newcommand{\QQ}{\mathbb Q}

\newcommand{\floor}[1]{\left\lfloor{#1}\right\rfloor}

\newcommand{\op}[1]{\operatorname{#1}}

\newcommand{\id}{\operatorname{id}}

\newcommand{\into}{\hookrightarrow}
\newcommand{\onto}{\twoheadrightarrow}

\newtheorem{theorem}{Theorem}[section]
\newtheorem{proposition}[theorem]{Proposition}
\newtheorem{lemma}[theorem]{Lemma}
\newtheorem{corollary}[theorem]{Corollary}
\newtheorem{example}[theorem]{Example}
\newtheorem{remark}[theorem]{Remark}
\newtheorem{definition}[theorem]{Definition}

\numberwithin{equation}{section}

\title{Generalized Periodicity in Group Cohomology}
\author{Nir Elber}
\address{University of California, Berkeley}
\email{nire@berkeley.edu}
\date{\today}

\begin{document}

\maketitle

\begin{abstract}
	\noindent Given a finite group $G$, we introduce ``encoding pairs,'' which are a pair of $G$-modules $M$ and $M'$ equipped with a shifted natural isomorphism between the cohomological functors $H^\bullet(G,\op{Hom}_\ZZ(M,-))$ and $H^\bullet(G,\op{Hom}_\ZZ(M',-))$. Studying these encoding pairs generalizes the theory of periodic cohomology for finite groups, allowing us to generalize the cohomological input of a theorem due to Swan that roughly says that a finite group with periodic cohomology acts feely on some sphere.
\end{abstract}

\setcounter{tocdepth}{4}
\tableofcontents

\section{Introduction}

Throughout this paper, $G$ will be a finite group, and $\widehat H^\bullet(G,-)$ denotes Tate cohomology. When $G=\langle\sigma\rangle$ is a cyclic group of order $n$, it is an amazing feature that there is some $\chi\in\widehat H^2(G,\ZZ)$ granting isomorphisms
\begin{equation}
	(-\cup\chi)\colon\widehat H^0(G,M)\to\widehat H^2(G,M) \label{eq:motivatingcycliccohom}
\end{equation}
for any $G$-module $M$. In fact, it is not too hard to write down $\chi$ as being represented by the ``carrying'' $2$-cocycle
\[\left(\sigma^i,\sigma^j\right)\mapsto\floor{\frac{i+j}n},\]
so \eqref{eq:motivatingcycliccohom} is telling us that we can represent each cohomology class of $\widehat H^2(G,M)$ by a $2$-cocycle of the form
\[\left(\sigma^i,\sigma^j\right)\mapsto\floor{\frac{i+j}n}\alpha\]
for some $\alpha\in M^G$. This ``classification'' of $2$-cocycles in $\widehat H^2(G,M)$
makes the cohomology of cyclic groups easy to work with computationally. For example, it becomes relatively straightforward to classify group extensions of cyclic groups up to isomorphism.

From one perspective, this classification of classes of $2$-cocycles 
for cyclic groups says that we can retrieve all $2$-cocycles by keeping track of the single element $\alpha\in M^G=H^0(G,M)$, modulo some equivalence relation coming from Tate cohomology. The algebraic way to choose a single element of $M^G$ is by elements in
\[\op{Hom}_\ZZ\left(\ZZ,M^G\right)=\op{Hom}_{\ZZ[G]}(\ZZ,M)=\op{Hom}_\ZZ(\ZZ,M)^G.\]
As such, one can phrase \eqref{eq:motivatingcycliccohom} as providing a natural isomorphism of cohomology functors
\[(-\cup\chi)\colon\widehat H^0(G,\op{Hom}_\ZZ(\ZZ,-))\Rightarrow\widehat H^2(G,-).\]
Roughly speaking, we see that the $G$-module $\ZZ$ is doing some ``encoding'': the data of a $2$-cocycle class in $\widehat H^2(G,M)$ is directly translated into an element of $\widehat H^0(G,\op{Hom}_\ZZ(\ZZ,M))$. This Tate cohomology group is just made of morphisms $\ZZ\to M$ (fixed by the $G$-action) modulo an equivalence relation coming from the definition of Tate cohomology.

More generally, permit $G$ to be non-cyclic, and fix some $G$-module $\ZZ[G]^m/I$ for some $m\ge0$ and $G$-submodule $I\subseteq\ZZ[G]^m$. Now, an isomorphism
\[\widehat H^0\left(G,\op{Hom}_\ZZ(\ZZ[G]^m/I,M)\right)\to\widehat H^2(G,M),\]
for some $G$-module $M$ is still encoding $2$-cocycles into morphisms, where these morphisms look like
\[\widehat H^0\left(G,\op{Hom}_\ZZ(\ZZ[G]^m/I,M)\right)=\frac{\op{Hom}_{\ZZ[G]}\left(\ZZ[G]^m/I,M\right)}{N_G\op{Hom}_\ZZ(\ZZ[G]^m/I,M)}.\]
To see the encoding here, view the numerator as picking 
an $m$-tuple of elements of $M$ in the same way that we would choose morphisms $\op{Hom}_{\ZZ[G]}\left(\ZZ[G]^m,M\right)$, but we cannot 
just choose any $m$ elements because they must satisfy some relations dictated by $I$. Then the denominator provides an equivalence relation of the $m$-tuples which determines if two tuples live in the same ``class.''

The above discussion is intended to motivate the following definition.
\begin{definition}
	Let $G$ be a finite group and $r\in\ZZ$. 
	Then a pair of $G$-modules $(X,X')$ is an \textit{$r$-encoding pair} if and only if there is a natural isomorphism of functors 
	\[\widehat H^i(G,\op{Hom}_\ZZ(X,-))\simeq\widehat H^{i+r}(G,\op{Hom}_\ZZ(X',-))\]
	for some $i\in\ZZ$. If $X'=\ZZ$, we call $X$ an \textit{$r$-encoding module}. 
\end{definition}
\begin{remark}
    In the case where $X$ is $\ZZ$-free and $i>0$, we have $\widehat H^i(G,\op{Hom}_\ZZ(X,-))\simeq\op{Ext}^i_{\ZZ[G]}(X,-)$ by \cite[Proposition~III.2.2]{brown-cohomology}. For this reason, we will occasionally specialize to the case where $X$ and $X'$ are $\ZZ$-free, but we will avoid doing so when possible.
\end{remark}
The index $i\in\ZZ$ is more or less irrelevant in the above definitions. Indeed, our main theorem shows the following equivalences.
\begin{theorem} \label{thm:main}
	Let $G$ be a finite group. Given $G$-modules $X$ and $X'$ and $r\in\ZZ$, the following are equivalent.
	\begin{enumerate}
		\item[(a)] $(X,X')$ is an $r$-encoding pair.
		\item[(b)] If $r\ge0$, then $X$ is cohomologically equivalent to $X'\otimes_\ZZ I_G^{\otimes r}$. If $r<0$, then $X$ is cohomologically equivalent to $X'\otimes_\ZZ\op{Hom}_\ZZ(I_G^{\otimes r},\ZZ)$. (``Cohomological equivalence'' will be defined in \Cref{subsec:cohom-equiv}.)
		\item[(c)] There is $x\in\widehat H^r(G,\op{Hom}_\ZZ(X',X))$ granting a natural isomorphism
		\[(-\cup x)\colon\widehat H^i(G,\op{Hom}_\ZZ(X,-))\Rightarrow\widehat H^{i+r}(G,\op{Hom}_\ZZ(X',-))\]
		for all $i\in\ZZ$.
		\item[(d)] There are $x\in\widehat H^r(G,\op{Hom}_\ZZ(X',X))$ and $x'\in\widehat H^{-r}(G,\op{Hom}_\ZZ(X,X'))$ such that
		\[x'\cup x=[{\id_{X'}}]\in\widehat H^0(G,\op{Hom}_\ZZ(X',X'))\qquad\text{and}\qquad x\cup x'=[{\id_X}]\in\widehat H^0(G,\op{Hom}_\ZZ(X,X)).\]
	\end{enumerate}
	If $X$ and $X'$ are finitely generated $G$-modules, (a) is equivalent to each of
	the following.
	\begin{enumerate}
	    \item[(e)] For some $i\in\ZZ$, there is a natural isomorphism
		\[\widehat H^i(G,X'\otimes_\ZZ-)\Rightarrow\widehat H^{i+r}(G,X\otimes_\ZZ-).\]
		\item[(f)] There is $x\in\widehat H^r(G,\op{Hom}_\ZZ(X',X))$ granting natural isomorphisms
		\[(x\cup-)\colon\widehat H^i(G,X'\otimes_\ZZ-)\Rightarrow\widehat H^{i+r}(G,X\otimes_\ZZ-)\]
		for all $i\in\ZZ$.
	\end{enumerate}
	If $X$ is $\ZZ$-free and $X'=\ZZ$, (a) is equivalent to
	\begin{enumerate}
		\item[(g)] $\widehat H^r(G,X)\cong\ZZ/|G|\ZZ$ 
		and $\widehat H^0(G,\op{Hom}_\ZZ(X,X))$ is cyclic.
	\end{enumerate}
\end{theorem}
\begin{proof}
	The equivalence of (a), (c), and (d) are from \Cref{prop:better-encoding-pair}. The equivalence of (a) and (b) follows from combining \Cref{lem:encoding-pair-operations,ex:ig-encodes,ex:allencoderesexist}. The equivalence of (a), (e), and (f) follows from \Cref{prop:encode-to-tensor,prop:tensor-to-encode}.
	Lastly, the equivalence of (a) and (g) follows from \Cref{prop:finitecohomcheck}.
\end{proof}
\begin{remark}
    The ``twisting'' in part (e) of \Cref{thm:main} is akin to the twisting introduced in \cite{twisted-periodic-cohomology}.
\end{remark}

When we may take $X=X'=\ZZ$ (for example, when $G$ is cyclic), we are essentially studying groups with periodic cohomology, so many results and arguments in this paper will mimic those of periodic cohomology. However, periodic cohomology requires somewhat stringent conditions on the group itself, and allowing this ``free parameter'' $X$ will permit general groups at the cost of a perhaps more complex $X$. For example, when $r\ge0$, we can take $X=I_G^{\otimes r}$ and $X'=\ZZ$ for any finite group $G$ (as seen from \Cref{thm:main}), though this $G$-module is quite rough to handle.

We note that encoding pairs always arise ``for a good reason.'' Rigorously, in the $\ZZ$-free case, encoding pairs always arise from projective resolutions, in the following sense.
\begin{theorem} \label{thm:reason-for-encoding}
    Let $G$ be a group, and let $X$ and $X'$ be $\ZZ$-free $G$-modules. Given a positive integer $r$, the pair $(X,X')$ is an $r$-encoding pair if and only if there is an exact sequence
    \begin{equation}
        0\to X\to P_r\to P_{r-1}\to\cdots\to P_2\to P_1\to X'\to0, \label{eq:desired-es}
    \end{equation}
    where each $P_i$ is $\ZZ[G]$-projective.
\end{theorem}
We prove \Cref{thm:reason-for-encoding} in \Cref{subsec:free-res}.
\begin{remark}
    \Cref{thm:reason-for-encoding} is shown in the case of $X=X'=\ZZ$ in \cite[Theorem~4.1]{swan-periodic-resolutions}. This result is the main cohomological input to of Swan's \cite[Theorem~A]{swan-periodic-resolutions} that says any finite group with periodic cohomology acts freely on a finite simplicial complex with the same homotopy type as a sphere. Note that groups with free actions on spheres automatically have periodic cohomology; see, for example, \cite[p.~154]{brown-cohomology}.
\end{remark}

\section{General Shifting}

In this section, we establish some abstract results about shifting various functors and then apply them to begin the theory of encoding pairs.

\subsection{Shiftable Functors}
The main point of this subsection is to set up theory around ``shiftable functors,'' whose main application will be in the proofs of \Cref{cor:cupup,cor:cupdown}. To introduce this definition, we require the notion of an induced module.
\begin{definition}
    Fix a finite group $G$. Then a $G$-module $M$ is \textit{induced} if and only if $M\cong\op{Hom}_{\ZZ}(\ZZ[G],A)$ or $M\cong\ZZ[G]\otimes_\ZZ A$ for some abelian group $A$. Note that $\op{Hom}_{\ZZ}(\ZZ[G],A)\cong\ZZ[G]\otimes_\ZZ A$ by \cite[Proposition~III.5.9]{brown-cohomology}.
\end{definition}
And now here is the definition of a shiftable functor.
\begin{definition}
    Let $G$ be a finite group. Then a covariant functor $F\colon\op{Mod}_G\to\op{Mod}_G$ is a \textit{shiftable functor} if and only if $F$ is satisfies the following properties.
    \begin{itemize}
        \item The functor $F$ sends $\ZZ$-split short exact sequences of $G$-modules to short exact sequences of $G$-modules.
        \item The functor $F$ sends induced modules to induced modules.
    \end{itemize}
\end{definition}
The main point to shiftable functors $F$ is that the dimension-shifting $\ZZ$-split short exact sequences
\[\arraycolsep=1.4pt\begin{array}{ccccccccc}
	0 &\to& I_G\otimes_\ZZ A &\to& \ZZ[G]\otimes_\ZZ A &\to& A &\to& 0 \\
	0 &\to& A &\to& \op{Hom}_\ZZ(\ZZ[G],A) &\to& \op{Hom}_\ZZ(I_G,A) &\to& 0
\end{array}\]
will remain exact upon applying $F$, and the middle term will remain induced.
\begin{example}
    Let $G$ be a finite group and $X$ a $G$-module. Then the functors $\op{Hom}_\ZZ(X,-)$ and $X\otimes_\ZZ-$ are shiftable. We will check that $\op{Hom}_\ZZ(X,-)$ is shiftable and omit the check for $X\otimes_\ZZ-$.
    \begin{itemize}
        \item This functor is additive over $\ZZ$-modules, so it will take a $\ZZ$-split short exact sequence of $G$-modules to another $\ZZ$-split short exact sequence of $G$-modules.
        \item Given an induced module $M=\op{Hom}_\ZZ(\ZZ[G],A)$, then
    	\[\op{Hom}_\ZZ(X,M)\cong\op{Hom}_\ZZ(X,\op{Hom}_\ZZ(\ZZ[G],A))\cong\op{Hom}_\ZZ(X\otimes_\ZZ\ZZ[G],A)\cong\op{Hom}_\ZZ(\ZZ[G],\op{Hom}_\ZZ(X,A)),\]
    	so $\op{Hom}_\ZZ(X,M)$ is induced.
    \end{itemize}
\end{example}
A key property of shiftable functors is how we will be able to relate them to each other via cup products. With this in mind, we have the following definition.
\begin{definition}
	Let $G$ be a finite group. Then we define a \textit{shifting pair} $(F,F',X,\eta)$
	to be a pair of shiftable functors $F$ and $F'$ equipped with a $G$-module $X$ and a natural transformation
	\[\eta\colon F\otimes_\ZZ X\Rightarrow F'.\]
\end{definition}
\begin{example} \label{ex:shiftingpair}
	Given $G$-modules $X$ and $X'$, there is a canonical composition map
	\[\arraycolsep=1.4pt\begin{array}{crllcc}
		\eta\colon&\op{Hom}_\ZZ(X,-)&\otimes_\ZZ&\op{Hom}_\ZZ(X',X)&\Rightarrow&\op{Hom}_\ZZ(X',-) \\
		& f &\otimes& \varphi &\mapsto& f\circ\varphi
	\end{array}\]
	so $(\op{Hom}_\ZZ(X,-),\op{Hom}_\ZZ(X',-),\op{Hom}_\ZZ(X',X),\eta)$ is a shifting pair.
\end{example}
Having a natural transformation lets us define natural cup-product maps. In particular, given a shifting pair $(F,F',X,\eta)$, a $G$-module $A$, and integers $r,s\in\ZZ$, we can write down the composite map
\begin{equation}
    \widehat H^r(G,X)\otimes_\ZZ\widehat H^s(G,FA)\stackrel\cup\to\widehat H^{r+s}(G,X\otimes FA)\stackrel\eta\to\widehat H^{r+s}(G,F'A). \label{eq:cup-prod-induced-by-eta}
\end{equation}
We call this composite the ``cup-product map induced by $\eta$.''
\begin{lemma} \label{lem:cuppingisnatural}
	Let $G$ be a finite group, and let $(F,F',X,\eta)$ be a shifting pair. Then, given $r,s\in\ZZ$ and $c\in\widehat H^r(G,X)$, the cup-product maps
	\[(-\cup c)\colon\widehat H^s(G,F-)\Rightarrow\widehat H^{r+s}(G,F'-)\]
	induced by $\eta$ make a natural transformation of cohomology functors.
\end{lemma}
\begin{proof}
	Given a $G$-module $A$, we note that each morphism in the cup-product map defined in \eqref{eq:cup-prod-induced-by-eta} is natural, so we are done.
\end{proof}
Now, let's begin with a key result on shiftable functors. The proof gives a taste for why our hypotheses are chosen.
\begin{proposition} \label{prop:dimshiftcupisos}
	Let $G$ be a finite group, and let $(F,F',X,\eta)$ be a shifting pair. If we have $r,s\in\ZZ$ and $c\in H^r(G,X)$ such that the cup-product map
	\[(-\cup c)\colon\widehat H^s(G,F-)\Rightarrow\widehat H^{r+s}(G,F'-)\]
	is a natural isomorphism, then the cup-product map
	\[(-\cup c)\colon\widehat H^j(G,F-)\Rightarrow\widehat H^{r+j}(G,F'-)\]
	is a natural isomorphism for all $j\in\ZZ$.
\end{proposition}
\begin{proof}
	These cup-products are natural already by \Cref{lem:cuppingisnatural}, so we focus on showing that the component morphisms are isomorphisms. This is by dimension-shifting on $j$. We will show how to shift downwards. Shifting upwards is analogous.
	
	Suppose that the cup-product map is always an isomorphism for $j$, and we show that it is always an isomorphism for $j-1$. Namely, fix a $G$-module $A$, and we claim
	\[(-\cup c)\colon\widehat H^{j-1}(G,FA)\to\widehat H^{r+j-1}(G,F'A)\]
	is an isomorphism. To do so, we note we have $\ZZ$-split short exact sequence
	\[0\to I_G\otimes_\ZZ A\to\ZZ[G]\otimes_\ZZ A\to A\to0.\]
	Applying $F$ and $F'$, the definition of a shiftable functor yields the short exact sequences
	\[\arraycolsep=1.4pt\begin{array}{ccccccccc}
	    0 &\to& F(I_G\otimes_\ZZ A) &\to& F(\ZZ[G]\otimes_\ZZ A) &\to& FA &\to& 0, \\
	    0 &\to& F'(I_G\otimes_\ZZ A) &\to& F'(\ZZ[G]\otimes_\ZZ A) &\to& F'A &\to& 0.
	\end{array}\]
	The middle terms are still induced, so the boundary morphisms are isomorphisms, allowing us to see that
	\[\begin{tikzcd}
		{\widehat H^{j-1}(G,FA)} & {\widehat H^{r+j-1}(G,F'A)} \\
		{\widehat H^{j}(G,F(I_G\otimes_\ZZ A))} & {\widehat H^{r+j}(G,F'(I_G\otimes_\ZZ A))}
		\arrow["{(-\cup c)}", from=1-1, to=1-2]
		\arrow["{(-\cup c)}", from=2-1, to=2-2]
		\arrow["{\delta}"', from=1-1, to=2-1]
		\arrow["(-1)^r\delta"', from=1-2, to=2-2]
	\end{tikzcd}\]
	commutes because cup products commute with boundary homomorphisms \cite[p.~138]{brown-cohomology}. This finishes because the bottom row is an isomorphism by the inductive hypothesis, and the left and right morphisms are isomorphisms as discussed above.
\end{proof}
Here are some applications.
\begin{corollary} \label{cor:cupup}
	Let $G$ be a finite group. There exists $c\in\widehat H^1(G,I_G)$ such that, for any $G$-module $X$,
	\[(-\cup c)\colon\widehat H^i(G,\op{Hom}_\ZZ(X,-))\Rightarrow\widehat H^{i+1}(G,\op{Hom}_\ZZ(X,-\otimes_\ZZ I_G))\]
	is a natural isomorphism for any $i\in\ZZ$. In the future, we may label this cup-product map as $(-\cup c)_d$ because the simpler functor has smaller index.
\end{corollary}
\begin{proof}
	Here, we are using the shifting pair $(\op{Hom}_\ZZ(X,-),\op{Hom}_\ZZ(X,I_G\otimes_\ZZ-),I_G,\eta)$, where
	\[\eta_A\colon\op{Hom}_\ZZ(X,A)\otimes_\ZZ I_G\to\op{Hom}_\ZZ(X,A\otimes_\ZZ I_G)\]
	is the canonical map sending $f\otimes z$ to $x\mapsto (f(x)\otimes z)$.

	Now, in light of \Cref{prop:dimshiftcupisos}, we merely have to find $c\in\widehat H^1(G,I_G)$ such that $(-\cup c)$ is a natural isomorphism at $i=0$.
	Well, we note that we have the $\ZZ$-split short exact sequence
	\begin{equation}
	    0\to\op{Hom}_\ZZ(X,A\otimes_\ZZ I_G)\to\op{Hom}_\ZZ(X,A\otimes_\ZZ\ZZ[G])\to\op{Hom}_\ZZ(X,A)\to0 \label{eq:cup-up-ses}
	\end{equation}
	which will induce a boundary morphism
	\[\delta\colon\widehat H^0(G,\op{Hom}_\ZZ(X,A))\to\widehat H^1(G,\op{Hom}_\ZZ(X,A\otimes_\ZZ I_G)),\]
	which is an isomorphism because the middle term of \eqref{eq:cup-up-ses} is induced.
	
	To finish, we claim that $\delta$ equals the cup-product map $(-\cup c)$, where $c\in\widehat H^1(G,I_G)$ is represented by the $1$-cocycle $g\mapsto(g-1)$.
	Indeed, given $[f]\in\widehat H^0(G,\op{Hom}_\ZZ(X,A))$ where $f\colon X\to A$ is a $G$-module homomorphism, we can pull this back to the $0$-cochain $\widetilde f\colon X\to\ZZ[G]\otimes_\ZZ A$ defined by $\widetilde f(x)\coloneqq1\otimes f(x)$. Applying the differential, we can compute the $1$-cocycle $d\widetilde f\in B^1(G,\op{Hom}_\ZZ(X,\ZZ[G]\otimes_\ZZ A))$ is
	\begin{align*}
		(d\widetilde f)(g)(x) = (g-1)\otimes f(x),
	\end{align*}
	finishing.
\end{proof}
\begin{corollary} \label{cor:cupdown}
	Let $G$ be a finite group. There exists $c\in\widehat H^1(G,I_G)$ such that, for any $G$-module $X$,
	\[(-\cup c)\colon\widehat H^i(G,\op{Hom}_\ZZ(X,\op{Hom}_\ZZ(I_G,-)))\Rightarrow\widehat H^{i+1}(G,\op{Hom}_\ZZ(X,-))\]
	is a natural isomorphism for any $i\in\ZZ$. In the future, we may label this cup-product map as $(-\cup c)_u$ because the simpler functor has larger index.
\end{corollary}
\begin{proof}
	Similar to before, we are using the shifting pair $(\op{Hom}_\ZZ(X,\op{Hom}_\ZZ(I_G,-)),\op{Hom}_\ZZ(X,-),I_G,\eta)$, where
	\[\eta_A\colon\op{Hom}_\ZZ(X,\op{Hom}_\ZZ(I_G,A))\otimes_\ZZ I_G\Rightarrow\op{Hom}_\ZZ(X,A)\]
	is the canonical map sending $f\otimes z$ to $x\mapsto f(x)(z)$. The proof now proceeds in the same way as \Cref{cor:cupup} using the $\ZZ$-split short exact sequence
	\[0\to\op{Hom}_\ZZ(X,A)\to\op{Hom}_\ZZ(X,\op{Hom}_\ZZ(\ZZ[G],A))\to\op{Hom}_\ZZ(X,\op{Hom}_\ZZ(I_G,A))\to0,\]
    which will induce a boundary morphism
	\[\delta\colon\widehat H^0(G,\op{Hom}_\ZZ(X,\op{Hom}_\ZZ(I_G,A)))\to\widehat H^1(G,\op{Hom}_\ZZ(X,A)),\]
	which is an isomorphism because our middle term $\op{Hom}_\ZZ(X,\op{Hom}_\ZZ(\ZZ[G],A))$ is induced.

	Now, we claim that $\delta$ equals the cup-product map $(-\cup c)$ where $c\in\widehat H^1(G,I_G)$ is represented by the $1$-cocycle $g\mapsto(1-g)$. Indeed, fix some $[f]\in\widehat H^0(G,\op{Hom}_\ZZ(X,\op{Hom}_\ZZ(I_G,A)))$ where $f\colon X\to\op{Hom}_\ZZ(I_G,A)$ is a $G$-module homomorphism. This pulls back to the $0$-cochain $\widetilde f\in\widehat H^0(G,\op{Hom}_\ZZ(X,\op{Hom}_\ZZ(\ZZ[G],A)))$ defined by $\widetilde f(x)(z)\coloneqq f(x)(z-\varepsilon(z))$. Applying the differential, a short computation finds
	\begin{align*}
		(d\widetilde f)(g)(x)(z) = \varepsilon(z)f(x)\left(1-g\right).
	\end{align*}
	Thus, this pulls back to the $1$-cocycle $g\mapsto(x\mapsto f(x)(1-g))$ in $\widehat H^1(G,\op{Hom}_\ZZ(X,A))$, as needed.
\end{proof}
\begin{remark}
	The same proofs for \Cref{cor:cupup,cor:cupdown} will work when $\op{Hom}_\ZZ(X,-)$ is replaced by $X\otimes_\ZZ-$ or any composite of these. There is not an analogue for arbitrary shiftable functors because, for example, there is no obvious way to construct $\eta$ in general. Regardless, we will not need to work in these levels of generality.
\end{remark}
The point of \Cref{cor:cupup,cor:cupdown} is that we can now dimension-shift via cup products. In fact, the proofs show that we can use the same $c\in\widehat H^1(G,I_G)$ represented by $g\mapsto(g-1)$ for both shifting isomorphisms.

\subsection{Shifting Natural Transformations}
Observe that a natural transformation $F\Rightarrow F'$ of shiftable functors will induce natural transformations in cohomology
\[\widehat H^i(G,F-)\Rightarrow\widehat H^i(G,F'-)\]
It will turn out that, when $F=\op{Hom}_\ZZ(X,-)$ and $F'=\op{Hom}_\ZZ(X',-)$, all natural transformations in cohomology will come from cup products.

To begin, we show this result for $i=0$.
\begin{lemma} \label{lem:naturaltransiscupping}
	Let $G$ be a finite group, and let $X$ and $X'$ be $G$-modules. Suppose that, for some $r\in\ZZ$, there is a natural transformation
	\[\Phi_\bullet\colon\widehat H^0(G,\op{Hom}_\ZZ(X,-))\Rightarrow\widehat H^r(G,\op{Hom}_\ZZ(X',-)).\]
	Then there exists a unique $x\in\widehat H^r(G,\op{Hom}_\ZZ(X',X))$ such that $\Phi_\bullet=(-\cup x)$, where the cup product is induced by the shifting pair of \Cref{ex:shiftingpair}.
\end{lemma}
\begin{proof}
	This is essentially the Yoneda lemma. For uniqueness, we note having $x\in\widehat H^r(G,\op{Hom}_\ZZ(X,',X))$ such that $\Phi_\bullet=(-\cup x)$ lets us compute
	\[\Phi_X([{\id_X}])=[{\id_X}]\cup x=x,\]
	so we see the $x$ is forced. Now, for existence, set $x\coloneqq\Phi_X([\id_X])$. The point is to fix some $G$-module $A$ and any $[f]\in\widehat H^0(G,\op{Hom}_\ZZ(X,A))$ in order to track through the commutativity of the following diagram.
	\begin{equation*}
		\begin{tikzcd}
			{\widehat H^0(G,\op{Hom}_\ZZ(X,X))} & {\widehat H^r(G,\op{Hom}_\ZZ(X',X))} \\
			{\widehat H^0(G,\op{Hom}_\ZZ(X,A))} & {\widehat H^r(G,\op{Hom}_\ZZ(X',A))}
			\arrow["{\Phi_X}", from=1-1, to=1-2]
			\arrow["f"', from=1-1, to=2-1]
			\arrow["f"', from=1-2, to=2-2]
			\arrow["{\Phi_A}", from=2-1, to=2-2]
		\end{tikzcd}
	\end{equation*}
	Plugging in $[{\id_X}]$ in the above diagram, one can compute that $f([{\id_X}])=[f]$ and $f(x)=[f]\cup x$, so $\Phi_A([f])=[f]\cup x$ follows.
\end{proof}
\begin{example} \label{ex:ig-shifting-up}
    Let $G$ be a finite group. The short exact sequence
    \[0\to\op{Hom}_\ZZ(\ZZ,A)\to\op{Hom}_\ZZ(\ZZ[G],A)\to\op{Hom}_\ZZ(I_G,A)\to0\]
    induces a boundary isomorphism $\widehat H^0(G,\op{Hom}_\ZZ(I_G,A))\to\widehat H^{1}(G,\op{Hom}_\ZZ(\ZZ,A))$, and these assemble into a natural isomorphism. We expect this natural isomorphism to arise from a cup product by \Cref{lem:naturaltransiscupping}. Indeed, a computation similar to \Cref{cor:cupdown} reveals that the natural isomorphism $\widehat H^0(G,\op{Hom}_\ZZ(I_G,-))\Rightarrow\widehat H^{1}(G,\op{Hom}_\ZZ(\ZZ,-))$ is given by the cup product with the class $[c]\in\widehat H^1(G,I_G)$ defined by $c(g)\coloneqq 1-g$.
\end{example}
We now get the main result by dimension-shifting.
\begin{proposition} \label{prop:allnaturaltransarecups}
	Let $G$ be a finite group, and let $X$ and $X'$ be $G$-modules. Then, given $r,s\in\ZZ$, any natural transformation
	\[\Phi_\bullet^{(s)}\colon\widehat H^s(G,\op{Hom}_\ZZ(X,-))\Rightarrow\widehat H^{r+s}(G,\op{Hom}_\ZZ(X',-)),\]
	is $\Phi_\bullet^{(s)}=(-\cup x)$ for a unique $x\in\widehat H^r(G,\op{Hom}_\ZZ(X',X))$.
\end{proposition}
\begin{proof}
	This argument is by dimension-shifting the $s$ upwards and downwards. Namely, we show the conclusion of the statement by induction on $s$. For $s=0$, this is \Cref{lem:naturaltransiscupping}. We will explain how to construct $x$ in an induction upwards and why $x$ is unique in an induction downwards. The other inductions are similar. For brevity, we set $F\coloneqq\op{Hom}_\ZZ(X,-)$ and $F'\coloneqq\op{Hom}_\ZZ(X',-)$.

	We begin with the induction upwards for the construction of $x$. Suppose $x$ always exists for $s=i$, and we show that it exists for $s=i+1$, so fix a natural transformation
	\[\Phi_\bullet^{(i+1)}\colon\widehat H^{i+1}(G,F-)\Rightarrow\widehat H^{r+i+1}(G,F'-),\]
	which we would like to know arises as $(-\cup x)$ for some $x\in\widehat H^r(G,\op{Hom}_\ZZ(X',X))$. In order to apply the inductive hypothesis, we need to construct a natural transformation $\Phi_\bullet^{(i)}$, so we do so. Fix some $G$-module $A$, and we use \Cref{cor:cupup} to set $c\in\widehat H^1(G,I_G)$ given by $g\mapsto(g-1)$ to give the isomorphisms
	\[\arraycolsep=1.4pt\begin{array}{ccccccccc}
		(-\cup c)_d\colon& \widehat H^i(G,FA) &\to& \widehat H^{i+1}(G,F(A\otimes_\ZZ I_G)) \\
		(-\cup c)'_d\colon& \widehat H^{r+i}(G,F'A) &\to& \widehat H^{r+i+1}(G,F'(A\otimes_\ZZ I_G))
	\end{array}\]
	fitting into the diagram
	\[\begin{tikzcd}
		{\widehat H^i(G,FA)} & {\widehat H^{i+1}(G,F(I_G\otimes_\ZZ A))} \\
		{\widehat H^{r+i}(G,F'A)} & {\widehat H^{r+i+1}(G,F'(I_G\otimes_\ZZ A))}
		\arrow["{(-\cup c)_d}", from=1-1, to=1-2]
		\arrow["{\Phi^{(i+1)}_{I_G\otimes_\ZZ A}}", from=1-2, to=2-2]
		\arrow["{(-\cup c)'_d}", from=2-1, to=2-2]
		\arrow[dashed, from=1-1, to=2-1]
	\end{tikzcd}\]
	where the horizontal arrows are isomorphisms. Thus, we induce a morphism
	\[\Phi_A^{(i)}\coloneqq((-\cup c)'_d)^{-1}\circ\Phi^{(i+1)}_{I_G\otimes_\ZZ A}\circ(-\cup c)_d,\]
	which we see assembles into a natural transformation $\Phi_\bullet^{(i)}$. The inductive hypothesis now tells us that $\Phi_\bullet^{(i)}=(-\cup x)$ for some $x\in\widehat H^p(G,\op{Hom}_\ZZ(X',X))$.
	
	We now need to use what we know about $\Phi_\bullet^{(i+1)}$ to talk about $\Phi_\bullet^{(i)}$, so we need to shift back in the other direction. Fix some $G$-module $A$, and we use \Cref{cor:cupdown} to give isomorphisms
	\[\arraycolsep=1.4pt\begin{array}{ccccccccc}
		(-\cup c)_u\colon&\widehat H^i(G,F(\op{Hom}_\ZZ(I_G,A)))&\to&\widehat H^{i+1}(G,FA) \\
		(-\cup c)_u'\colon&\widehat H^{r+i}(G,F(\op{Hom}_\ZZ(I_G,A)))&\to&\widehat H^{r+i+1}(G,FA)
	\end{array}\]
	fitting into the diagram
	\begin{equation}
		\begin{tikzcd}
			{\widehat H^i(G,F(\op{Hom}_\ZZ(I_G,A)))} & {\widehat H^{i+1}(G,FA)} \\
			{\widehat H^{r+i}(G,F'(\op{Hom}_\ZZ(I_G,A)))} & {\widehat H^{r+i+1}(G,F'A)}
			\arrow["{(-\cup c)_u}", from=1-1, to=1-2]
			\arrow["{(-\cup c)_u'}", from=2-1, to=2-2]
			\arrow["{(-\cup x)}"', from=1-1, to=2-1]
			\arrow[dashed, from=1-2, to=2-2]
		\end{tikzcd} \label{eq:shiftingcupup}
	\end{equation}
	which we would like to commute by filling in the dashed arrow. Indeed, the horizontal arrows are isomorphisms, so there is a unique dashed arrow making the diagram commute. On one hand, we see $\left(-\cup(-1)^ix\right)$ can fill the right arrow because any $a\in\widehat H^i(G,F(\op{Hom}_\ZZ(I_G,A)))$ has
	\[(a\cup x)\cup c=a\cup(x\cup c)=(-1)^ia\cup(c\cup x)=(a\cup c)\cup(-1)^ix.\]
	On the other hand, we claim that $\Phi_A^{(i+1)}$ can fill in the right arrow, which will finish by uniqueness of this arrow. For this, we draw the following very large diagram.
	\[\begin{tikzcd}
		{\widehat H^i(G,F(\op{Hom}_\ZZ(I_G,A)))} && {\widehat H^{i+1}(G,FA)} \\
		& {\widehat H^{i+1}(G,F(\op{Hom}_\ZZ(I_G,A)\otimes_\ZZ I_G))} \\
		\\
		{\widehat H^{r+i}(G,F'(\op{Hom}_\ZZ(I_G,A)))} && {\widehat H^{r+i+1}(G,F'A)} \\
		& {\widehat H^{r+i+1}(G,F'(\op{Hom}_\ZZ(I_G,A)\otimes_\ZZ I_G))}
		\arrow["{(-\cup c)_u}"{description, pos=0.2}, from=1-1, to=1-3]
		\arrow["{(-\cup c)_d}"{description}, from=1-1, to=2-2]
		\arrow["f"{description}, from=2-2, to=1-3]
		\arrow["{(-\cup x)}"{description, pos=0.3}, from=1-1, to=4-1]
		\arrow["{\Phi^{(i+1)}_{\op{Hom}_\ZZ(I_G,A)\otimes_\ZZ I_G}}"{description, pos=0.3}, from=2-2, to=5-2]
		\arrow["{\Phi^{(i+1)}_A}"{description, pos=0.3}, from=1-3, to=4-3]
		\arrow["{(-\cup c)_u'}"{description, pos=0.2}, from=4-1, to=4-3]
		\arrow["{(-\cup c)_d'}"{description}, from=4-1, to=5-2]
		\arrow["f"{description}, from=5-2, to=4-3]
	\end{tikzcd}\]
	Here, the $f$ maps are induced by the canonical evaluation map. We want the outer rectangle to commute, for which it suffices to show that each parallelogram and the small top and bottom triangles to commute.
	\begin{itemize}
		\item The left parallelogram commutes by definition of $\Phi_\bullet^{(i)}=(-\cup x)$.
		\item The right parallelogram commutes by naturality of $\Phi^{(i+1)}_\bullet$.
		\item The triangles commute because the two paths are both cup products $(-\cup c)$ followed by some evaluation maps, and we can compute that the evaluation maps are the same.
	\end{itemize}
	The above commutativity checks finishes the construction of $x$.
	
	It remains to induct downwards for the uniqueness of $x$. Well, suppose $x$ is unique for $s=i+1$, and we show uniqueness for $s=i$. Suppose we have two $x,x'\in\widehat H^r(G,\op{Hom}_\ZZ(X',X))$ such that $\Phi_\bullet^{(i)}=(-\cup x)=(-\cup x')$. However, using \eqref{eq:shiftingcupup}, we see that
	\[\Phi_\bullet^{(i+1)}=\big(-\cup(-1)^ix\big)=\big(-\cup(-1)^ix'\big),\]
	so the uniqueness for $i+1$ enforces $x=x'$.
\end{proof}

\section{Encoding Pairs}

In this section, we establish the basic theory of encoding modules and how to build them. We will show the equivalence of (a)--(d) in \Cref{thm:main} as well as \Cref{thm:reason-for-encoding}.

\subsection{Encoding Elements}
In the theory of periodic cohomology (e.g., see \cite[Section~XII.11]{cartan-eilenberg}), it is helpful to phrase the theory in terms of having some elements $x\in\widehat H^r(G,\ZZ)$ and $y\in\widehat H^{-r}(G,\ZZ)$ such that
\[x\cup y=[1]\in\widehat H^0(G,\ZZ).\]
The generalization to $r$-encoding pairs $(X,X')$ is to turn $\widehat H^r(G,\ZZ)$ into $\widehat H^r(G,\op{Hom}_\ZZ(X',X))$ and use the elements provided by \Cref{prop:allnaturaltransarecups}. This idea gives the following result.
\begin{proposition} \label{prop:better-encoding-pair}
    Let $G$ be a finite group, and fix $G$-modules $X$ and $X'$ and $r\in\ZZ$. The following are equivalent.
    \begin{enumerate}
        \item[(a)] $(X,X')$ is an $r$-encoding pair.
        \item[(b)] There exists $x\in\widehat H^r(G,\op{Hom}_\ZZ(X',X))$ such that
        \[(-\cup x)\colon\widehat H^i(G,\op{Hom}_\ZZ(X,-))\to\widehat H^{i+r}(G,\op{Hom}_\ZZ(X',-))\]
        is a natural isomorphism for all $i\in\ZZ$.
        \item[(c)] There exist $x\in\widehat H^r(G,\op{Hom}_\ZZ(X',X))$ and $x'\in\widehat H^{-r}(G,\op{Hom}_\ZZ(X,X'))$ such that $x'\cup x=[{\id_{X'}}]$ and $x\cup x'=[{\id_{X}}]$.
    \end{enumerate}
\end{proposition}
\begin{proof}
    We show three implications.
    \begin{itemize}
        \item We show (a) implies (b). Note we are granted a natural isomorphism
        \[\Phi_\bullet\colon\widehat H^i(G,\op{Hom}_\ZZ(X,-))\to\widehat H^{i+r}(G,\op{Hom}_\ZZ(X',-))\]
        for some $i\in\ZZ$. By \Cref{prop:allnaturaltransarecups}, we can find $x\in\widehat H^r(G,\op{Hom}_\ZZ(X',X))$ such that $\Phi_\bullet=(-\cup x)$. However, this means $(-\cup x)$ is a natural isomorphism, so \Cref{prop:dimshiftcupisos} extends our isomorphism to all indices.
        \item We show (b) implies (c). To conjure $x'$, we note
        \[(-\cup x)\colon\widehat H^{-r}(G,\op{Hom}_\ZZ(X,-))\Rightarrow\widehat H^0(G,\op{Hom}_\ZZ(X',-))\]
        is a natural isomorphism and so has an inverse natural transformation. By \Cref{prop:allnaturaltransarecups}, we find $x'\in\widehat H^{-r}(G,\op{Hom}_\ZZ(X,X'))$ so that $(-\cup x')$ is the inverse.
        
        It remains to check the relations between $x$ and $x'$. The main point is that $(-\cup x)$ and $(-\cup x')$ are inverse natural isomorphisms by construction of $x'$. For example, the composite
        \[(-\cup x)\circ(-\cup x')=(-\cup x')\cup x=-\cup(x'\cup x)\]
        must be the identity natural transformation $\widehat H^0(G,\op{Hom}_\ZZ(X',-))\Rightarrow\widehat H^0(G,\op{Hom}_\ZZ(X',-))$. However, the identity is given by $-\cup[{\id_{X'}}]$, so we must have $x'\cup x=[{\id_{X'}}]$ by the uniqueness of \Cref{prop:allnaturaltransarecups}. The other check is similar.
        \item We show (c) implies (a). Working at index $0$, we claim that
    	\[\arraycolsep=1.4pt\begin{array}{rccc}
    		(-\cup x)\colon& \widehat H^0(G,\op{Hom}_\ZZ(X,-)) &\Rightarrow& \widehat H^{r}(G,\op{Hom}_\ZZ(X',-)) \\
    		(-\cup x')\colon& \widehat H^{r}(G,\op{Hom}_\ZZ(X',-)) &\Rightarrow& \widehat H^0(G,\op{Hom}_\ZZ(X,-))
    	\end{array}\]
    	are inverse natural transformations, which will finish. In one direction, we compute
    	\[(-\cup x)\circ(-\cup x')=(-\cup x')\cup x=-\cup(x'\cup x)=-\cup[{\id_{X'}}],\]
    	which is the identity. The other inverse check is similar.
    	\qedhere
    \end{itemize}
\end{proof}
\begin{example} \label{ex:ig-no-sign-error}
    It might appear odd that \Cref{prop:better-encoding-pair} has managed to assert that both $x\cup x'$ and $x'\cup x$ are identities even though $\cup$ might be anticommutative. Approximately speaking, this is possible because the maps
    \begin{equation}
        \arraycolsep=1.2pt\begin{array}{cl}
            \op{Hom}_\ZZ(X',X)\otimes_\ZZ\op{Hom}_\ZZ(X,X') &\to \op{Hom}_\ZZ(X,X) \\
            \op{Hom}_\ZZ(X,X')\otimes_\ZZ\op{Hom}_\ZZ(X',X) &\to \op{Hom}_\ZZ(X',X')
        \end{array} \label{eq:composition-map-for-cup}
    \end{equation}
    go in different orders, even in the case where $X=X'$. An example will be illuminating; continue from \Cref{ex:ig-shifting-up}, where we found some $[c]\in\widehat H^1(G,I_G)$. Now, define $f\colon I_G\to\ZZ$ by $f(g-1)=1$ for each $g\in G$ not equal to the identity $e\in G$. We check that $[f]\cup[c]=[1]$ and $[c]\cup[f]=[{\id_{I_G}}]$ for the suitably defined cup-product maps.
    \begin{itemize}
        \item We check $[f]\cup[c]=[1]$. Using \cite[Lemma~5.6]{bonn-lectures}, we can compute
        \[f\cup c=-\sum_{g\in G}(gf)(c(g))=-\sum_{g\in G}f\left(g^{-1}-1\right)=-\left(\left|G\right|-1\right),\]
        which is $[1]\in\widehat H^0(G,\ZZ)$ because Tate cohomology is $\left|G\right|$-torsion.
        \item We check $[c]\cup[f]=[{\id_{I_G}}]$. Again using \cite[Lemma~5.6]{bonn-lectures}, a computation shows that $(c\cup f)\in\op{Hom}_{\ZZ[G]}(I_G,I_G)$ is given by
        \[(c\cup f)(h-1)=\sum_{g\in G}c(g)\cdot(gf)(h-1)=-(h-1)\sum_{g\ne e}g,\]
        where $e\in G$ is the identity. Now, one can compute that the maps $I_G\to I_G$ which are right multiplication by an element of $I_G$ are in the image of the norm map on $\op{Hom}_{\ZZ}(I_G,I_G)$. As such, it follows that $[c]\cup[f]$ is the same class as $-\left(\left|G\right|-1\right)[{\id_{I_g}}]$, which is the same class as $[{\id_{I_G}}]$ because Tate cohomology is $\left|G\right|$-torsion.
    \end{itemize}
\end{example}
\begin{remark}
    It is perhaps notable that the case where $X=X'=\ZZ$ does make the maps in \eqref{eq:composition-map-for-cup} commute, so here we can in fact compare $x\cup x'$ with $x'\cup x$. This looks concerning when $\left|G\right|>2$, but it is actually still okay: what we see is that $x$ and $x'$ must live in even degrees for the cup product to commute. This is simply the statement that the cup product can only induce periodic cohomology with even periods; see \cite[Exercise~IV.9.1]{brown-cohomology}.
\end{remark}
The proof shows us that the elements $x$ and $x'$ store our ``encoding'' information, so they deserve names.
\begin{definition}
    Let $G$ be a finite group, and let $(X,X')$ be an $r$-encoding pair. Then elements $x\in\widehat H^r(G,\op{Hom}_\ZZ(X',X))$ and $x'\in\widehat H^{-r}(G,\op{Hom}_\ZZ(X,X'))$ satisfying $x'\cup x=[{\id_{X'}}]$ and $x\cup x'=[{\id_X}]$ are called an \textit{encoding element} and \textit{decoding element}, respectively.
\end{definition}
Given an $r$-encoding pair $(X,X')$, an encoding element $x$ is not uniquely determined from this data because the natural isomorphism is not given. However, the decoding element $x'$ is uniquely determined from $x$ because $x'$ is chosen so that $(-\cup x')$ is the inverse natural transformation of $(-\cup x)$. Thus, we are interested in investigating how unique an encoding element is.

To state this result correctly, we need to set up some structure. Let $X$ and $X'$ be $G$-modules. Note that $R\coloneqq\widehat H^0(G,\op{Hom}_\ZZ(X,X))$ is actually a ring with operation given by the cup product, which is the same as composition. Additionally, $R$ acts on $\widehat H^r(G,\op{Hom}_\ZZ(X',X))$ on the left by the cup product. Now, here is our result: encoding elements are unique up to an element in $R^\times$.
\begin{proposition} \label{prop:uniq-encoding}
    Let $G$ be a finite group, and let $(X,X')$ be an $r$-encoding pair. For brevity, set $R\coloneqq\widehat H^0(G,\op{Hom}_\ZZ(X,X))$ and $M\coloneqq\widehat H^r(G,\op{Hom}_\ZZ(X',X))$.
    \begin{enumerate}
        \item[(a)] Given two encoding elements $x,y\in M$, there is a unique $[\varphi]\in R^\times$ such that $y=[\varphi]\cup x$.
        \item[(b)] Given an encoding element $x\in M$ and some $[\varphi]\in R^\times$, then $y=[\varphi]\cup x$ is an encoding element.
    \end{enumerate}
\end{proposition}
\begin{proof}
    The idea here is that an encoding element $x$ will induce an isomorphism $(-\cup x)\colon R\to M$ of $R$-modules.
    \begin{enumerate}
        \item[(a)] We have two $R$-module isomorphisms $(-\cup x),(-\cup y)\colon R\to M$, so composing gives an isomorphism
        \[R\stackrel{(-\cup y)}\to M\stackrel{(-\cup x)}\leftarrow R\]
        of $R$-modules. Passing the identity $[{\id_X}]\in R$ through gives some $[\varphi]\in R$ such that $y=[\varphi]\cup x$, and this construction is unique. Lastly, note the above composite is $(-\cup[\varphi])\colon R\to R$, which is an $R$-module isomorphism, so it follows $[\varphi]\in R^\times$.
        \item[(b)] Letting $x'$ be the decoding element for $x$, we see $x'\cup[\varphi]^{-1}$ is the decoding element for $y$.
        \qedhere
    \end{enumerate}
\end{proof}
\begin{remark} \label{rem:encoding-module-ring}
    In some circumstances, we can explicitly compute the units of $R=\widehat H^0(G,\op{Hom}_\ZZ(X,X))$. For example, when $X'=\ZZ$, then $R\cong\ZZ/|G|\ZZ$ generated by $[{\id_X}]$. This will fall out of theory we establish later: by \Cref{cor:compute-cohom-pairs}, encoding elements $x$ give an isomorphism
    \[\widehat H^0(G,\op{Hom}_\ZZ(X,X))\stackrel{(-\cup x)}\to\widehat H^r(G,X)\stackrel{(x\cup-)}\leftarrow\widehat H^0(G,\ZZ)=\ZZ/|G|\ZZ\]
    of groups sending $[{\id_X}]$ to $1$, and it is quick to check that this is an isomorphism of rings.
\end{remark}

\subsection{Building Encoding Pairs} \label{subsec:make-pairs}
With \Cref{prop:better-encoding-pair} in hand, we can now make lots of encoding pairs. We begin with some basic module operations.
\begin{lemma} \label{lem:encoding-pair-operations}
	Let $G$ be a finite group, and fix an $r$-encoding pair $(X,X')$ and an $s$-encoding pair $(Y,Y')$.
	\begin{enumerate}
	    \item[(a)] $(X',X)$ is a $(-r)$-encoding pair.
	    \item[(b)] If $r=s$, then $(X\oplus Y,X'\oplus Y')$ is an $r$-encoding pair.
	    \item[(c)] $(X\otimes_\ZZ Y,X'\otimes_\ZZ Y')$ is an $(r+s)$-encoding pair.
	    \item[(d)] If $X'=Y$, then $(X,Y')$ is an $(r+s)$-encoding pair.
	\end{enumerate}
\end{lemma}
\begin{proof}
	Note (a) is direct from the definitions. Continuing, (b) follows by summing the granted natural isomorphisms at (say) index $0$, which is legal by \Cref{prop:better-encoding-pair}.
	
	For (c) and (d), note \Cref{prop:better-encoding-pair} gives natural isomorphisms
	\[\widehat H^0(G,\op{Hom}_\ZZ(X,-))\simeq\widehat H^r(G,\op{Hom}_\ZZ(X',-))\qquad\text{and}\qquad\widehat H^r(G,\op{Hom}_\ZZ(Y,-))\simeq\widehat H^{r+s}(G,\op{Hom}_\ZZ(Y',-)).\]
    Thus, to get (c), we combine and use the tensor--hom adjunction, writing
    \begin{align*}
        \widehat H^0(G,\op{Hom}_\ZZ(X\otimes_\ZZ Y,-)) &\simeq \widehat H^0(G,\op{Hom}_\ZZ(X,\op{Hom}_\ZZ(Y,-))) \\
        &\simeq \widehat H^r(G,\op{Hom}_\ZZ(X',\op{Hom}_\ZZ(Y,-))) \\
        &\simeq \widehat H^r(G,\op{Hom}_\ZZ(Y,\op{Hom}_\ZZ(X',-))) \\
        &\simeq \widehat H^{r+s}(G,\op{Hom}_\ZZ(Y',\op{Hom}_\ZZ(X',-))) \\
        &\simeq \widehat H^{r+s}(G,\op{Hom}_\ZZ(X'\otimes_\ZZ Y',-)).
    \end{align*}
	Lastly, for (d), we compose the natural isomorphisms by writing
    \[\widehat H^0(G,\op{Hom}_\ZZ(X,-))\simeq\widehat H^r(G,\op{Hom}_\ZZ(X',-))=\widehat H^r(G,\op{Hom}_\ZZ(Y,-))\simeq\widehat H^{r+s}(G,\op{Hom}_\ZZ(Y',-)),\]
    as desired.
\end{proof}
\begin{example} \label{ex:ig-encodes}
    For any $G$-module $A$, we have the dimension-shifting short exact sequence
    \[0\to\op{Hom}_\ZZ(\ZZ,A)\to\op{Hom}_\ZZ(\ZZ[G],A)\to\op{Hom}_\ZZ(I_G,A)\to0,\]
    so it follows $(I_G,\ZZ)$ is a $1$-encoding pair. (This will also follow from \Cref{prop:encodingses}.) In fact, for any $r\ge0$, we see $(I_G^{\otimes r},\ZZ)$ is an $r$-encoding pair.
    
    As an aside, \Cref{cor:cupdown} tells us that the $1$-cocycle $x\in Z^1(G,I_G)$ given by $x(g)=g-1$ represents an encoding element of the pair $(I_G,\ZZ)$. For the decoding element, fix any $g\in G\setminus\{e\}$, and one can compute that the homomorphism $x'\in\op{Hom}_\ZZ(I_G,\ZZ)$ given by $x'(h-1)\coloneqq-1_{g=h}$ has $x'\cup x=[1]$ and thus represents the decoding element.
\end{example}
Additionally, we note we can take duals.
\begin{proposition} \label{prop:dualmodule}
	Let $G$ be a finite group, and let $M$ be a $G$-module, and define the contravariant functor $(-)^\lor\coloneqq\op{Hom}_\ZZ(-,M)$. If $(X,Y)$ be an $r$-encoding pair, then $(Y^\lor,X^\lor)$ is an $r$-encoding pair.
\end{proposition}
\begin{proof}
	For any $G$-modules $A$ and $B$, note that functoriality of $(-)^\lor$ promises a morphism $\op{Hom}_\ZZ(A,B)\to\op{Hom}_\ZZ(B^\lor,A^\lor)$, explicitly by sending $f\colon A\to B$ to the map $f^\lor\colon g\mapsto(g\circ f)$. Further, functoriality tells us that
    \begin{equation}
        \begin{tikzcd}
        	{\op{Hom}_\ZZ(B,C)\otimes_\ZZ\op{Hom}_\ZZ(A,B)} & {\op{Hom}_\ZZ(A,C)} & {g\otimes f} & {g\circ f} \\
        	{\op{Hom}_\ZZ(C^\lor,B^\lor)\otimes_\ZZ\op{Hom}_\ZZ(B^\lor,A^\lor)} & {\op{Hom}_\ZZ(C^\lor,A^\lor)} & {g^\lor\otimes f^\lor} & {f^\lor\circ g^\lor}
        	\arrow[from=1-1, to=1-2]
        	\arrow[from=1-2, to=2-2]
        	\arrow[from=2-1, to=2-2]
        	\arrow[from=1-1, to=2-1]
        	\arrow[maps to, from=1-3, to=2-3]
        	\arrow[maps to, from=2-3, to=2-4]
        	\arrow[maps to, from=1-3, to=1-4]
        	\arrow[maps to, from=1-4, to=2-4]
        \end{tikzcd} \label{eq:dual-functoriality}
    \end{equation}
    commutes.
    
    We now proceed with the proof. We use \Cref{prop:better-encoding-pair} and dualize the elements in (c). Indeed, we have $x\in\widehat H^r(G,\op{Hom}_\ZZ(Y,X))$ and $y\in\widehat H^{-r}(G,\op{Hom}_\ZZ(X,Y))$ such that $x\cup y=[{\id_X}]$ and $y\cup x=[{\id_Y}]$. Passing $x$ through $(-)^\lor\colon\op{Hom}_\ZZ(Y,X)\to\op{Hom}_\ZZ(X^\lor,Y^\lor)$, we get some $x^\lor\in\widehat H^r(G,\op{Hom}_\ZZ(X^\lor,Y^\lor))$, and we construct $y^\lor$ similarly. Now, \eqref{eq:dual-functoriality} lets us compute
    \[x^\lor\cup y^\lor=(y\cup x)^\lor=[{\id_Y}]^\lor=[{\id_{Y^\lor}}].\]
    Checking $y^\lor\cup x^\lor=[{\id_{X^\lor}}]$ is similar and completes the proof.
\end{proof}
\begin{example} \label{ex:allencoderesexist}
	\Cref{ex:ig-encodes} established that $(I_G^{\otimes r},\ZZ)$ is an $r$-encoding pair for $r\ge0$. By \Cref{prop:dualmodule}, $\op{Hom}_\ZZ\left(I_G^{\otimes r},\ZZ\right)$ is a $(-r)$-encoding pair for $-r\le0$. Thus, we have established existence for $r$-encoding modules for all $r\in\ZZ$.
\end{example}
\begin{remark} \label{rem:encodingsubgroups} 
    More generally, any functor which commutes with cup products will also produce encoding pairs. For example, given a subgroup $H\subseteq G$, an $r$-encoding pair $(X,X')$ of $G$-modules yields an $r$-encoding pair of restricted $H$-modules $\left(\op{Res}X,\op{Res}X'\right)$. Indeed, using the encoding and decoding elements $x$ and $x'$, we can compute
    \[(\op{Res}x)\cup(\op{Res}x')=\op{Res}(x\cup x')=\op{Res}([{\id_X}])=[{\id_{\op{Res}X}}]\]
    and similar on the other side.
\end{remark}
\begin{example}
	It is not true that, if $X$ is an $r$-encoding $G_p$-module for all Sylow $p$-subgroups $G_p\subseteq G$, then $X$ is an $r$-encoding $G$-module. Indeed, take $X=\ZZ$ and $G=S_3$: all Sylow $p$-subgroups of $S_3$ are cyclic, so $\ZZ$ is a $2$-encoding module for all these subgroups. However, $S_3$ is not cyclic, so
	\[\widehat H^{-2}(G,\op{Hom}_\ZZ(X,\ZZ))\simeq\widehat H^{-2}(G,\ZZ)\simeq S_3/[S_3,S_3]\not\cong\ZZ/6\ZZ=\widehat H^0(G,\ZZ).\]
\end{example}
Taking duals might not look very impressive, but they essentially say that we do not have to look at torsion while studying encoding pairs because $\op{Hom}_\ZZ(\op{Hom}_\ZZ(X,\ZZ),\ZZ)$ has no torsion. Indeed, \Cref{prop:dualmodule} gives the following.
\begin{corollary} \label{cor:dual-twice}
    Let $G$ be a finite group, and define the contravariant functor $(-)^\lor\coloneqq\op{Hom}_\ZZ(-,\ZZ)$. If $(X,Y)$ is an $r$-encoding pair, then $(X^{\lor\lor},Y^{\lor\lor})$ is an $r$-encoding pair.
\end{corollary}
\begin{remark} \label{rem:dual-twice-fg}
    We continue in the context of \Cref{cor:dual-twice}. Given a $\ZZ[G]$-module $A$, we let $A_t$ denote the $\ZZ$-torsion subgroup. In the case where $A$ is finitely generated over $\ZZ$, the canonical morphism $A/A_t\to A^{\lor\lor}$ is an isomorphism. Applying this to our theory, we see that a finitely generated $r$-encoding modules $X$ naturally produces a $\ZZ$-free $r$-encoding module $X/X_t$.
\end{remark}

\subsection{Using Free Resolutions} \label{subsec:free-res}
Encoding pairs also behave in short exact sequences.
\begin{proposition} \label{prop:encodingses}
	Let $G$ be a finite group, and let
	\[0\to X\to M\to X'\to 0\]
	be a $\ZZ$-split short exact sequence such that $M$ is an induced $G$-module. Then $(X,X')$ is a $1$-encoding pair.
\end{proposition}
\begin{proof}
	Because the short exact sequence is $\ZZ$-split, we have the short exact sequence
	\[0\to\op{Hom}_\ZZ(X,A)\to\op{Hom}_\ZZ(M,A)\to\op{Hom}_\ZZ(X',A)\to0\]
	for any $G$-module $A$, which gives the isomorphism
	\[\delta_A\colon\widehat H^0(G,\op{Hom}_\ZZ(X',A))\to\widehat H^1(G,\op{Hom}_\ZZ(X,A))\]
	because $\op{Hom}_\ZZ(M,A)$ is induced. The $\delta_A$ assemble into a natural isomorphism $\delta_\bullet\colon\widehat H^0(G,\op{Hom}_\ZZ(X',-))\Rightarrow\widehat H^1(G,\op{Hom}_\ZZ(X,-))$, which is what we wanted.
\end{proof}
\begin{example}
	Fix a finite group $G$ generated by $S\coloneqq\langle\sigma_1,\ldots,\sigma_n\rangle$, and let $M\coloneqq\ZZ[G]^{\#S}$ have basis $\{e_i\}_{i=1}^m$. Then there is a projection $\pi\colon\ZZ[G]^{|G|}\onto I_G$ by sending $e_i\mapsto(\sigma_i-1)$, giving the short exact sequence
	\[0\to\ker\pi\to\ZZ[G]^{\#S}\to I_G\to0.\]
	This short exact sequence is $\ZZ$-split because $I_G$ is $\ZZ$-free. Because $\ZZ[G]^{\#S}\cong\ZZ[G]\otimes_\ZZ\ZZ^{\#S}$ is induced, we see $(\ker\pi,I_G)$ is a $1$-encoding pair by \Cref{prop:encodingses}. But $(I_G,\ZZ)$ is a $1$-encoding pair, so $(\ker\pi,\ZZ)$ is a $2$-encoding pair by \Cref{lem:encoding-pair-operations}.
\end{example}
We are now ready to prove \Cref{thm:reason-for-encoding}.
\begin{proof}[Proof of \Cref{thm:reason-for-encoding}]
    If such an exact sequence exists, then $(X,X')$ is an $r$-encoding pair by inductively applying \Cref{prop:encodingses}.
    
    It remains to show the converse. Note that $X$ and $X'$ are both $\ZZ$-free, so \cite[Proposition~III.2.2]{brown-cohomology} promises natural isomorphisms
    \[\op{Ext}^i_{\ZZ[G]}(X,-)\simeq H^i(G,\op{Hom}_\ZZ(X,-))\qquad\text{and}\qquad\op{Ext}^i_{\ZZ[G]}(X',-)\simeq H^i(G,\op{Hom}_\ZZ(X',-))\]
    for $i\in\NN$. If we let $x\in H^r(G,\op{Hom}_\ZZ(X,X'))$ denote an encoding element, then the fact that $(X,X')$ is an $r$-encoding pair implies that the induced cup-product maps
    \[(-\cup x)\colon\op{Ext}_{\ZZ[G]}^i(X,-)\Rightarrow\op{Ext}_{\ZZ[G]}^{i+r}(X',-)\]
    is an isomorphism for $i>0$ and epic (on $G$-modules) at $i=0$. The existence of the needed exact sequence now follows from \cite[Theorem~1.2]{wall-almost-conjecture}.
\end{proof}
We close this section with an example exhibiting the ``encoding'' mentioned in the introduction.
\begin{example}
    Let $G$ be a finite abelian group, where $G=\bigoplus_{i=1}^mG_i$, where $G_i=\langle\sigma_i\rangle$, where $\sigma_i\in G$ is an element of order $n_i$. Let $(W_i,d_i)$ be the free resolution of $\ZZ[G_i]$-modules given by
    \[W_i\colon\cdots\xrightarrow{\sigma_i-1}\ZZ[G_i]\xrightarrow{N_i}\ZZ[G_i]\xrightarrow{\sigma_i-1}\ZZ[G_i]\to\ZZ\to0,\]
    where $N_i=\sum_{k=0}^{n_i-1}\sigma_i^k$. Let $(W,d)$ be the tensor product complex of the $(W_i,d_i)$, which gives a $\ZZ[G]$-free resolution of $\ZZ$. \Cref{thm:reason-for-encoding} now tells us that $\ker d_r$ is an $r$-encoding module. Tracking this through at degree $2$ is able to recover abelian crossed products, in the sense of \cite{abelian-crossed}; see \cite[pp.~421--423]{cohom-abelian-crossed} for details.
\end{example}

\subsection{Cohomological Equivalence} \label{subsec:cohom-equiv}
It might be the case that many different $G$-modules $X$ are doing essentially the same encoding $\widehat H^0(G,\op{Hom}_\ZZ(X,-))\Rightarrow\widehat H^r(G,-)$, so it will be helpful to have language for this phenomenon.
\begin{definition}
	Let $G$ be a finite group. We say that two $G$-modules $X$ and $X'$ are \textit{cohomologically equivalent} if and only if $(X,X')$ is a $0$-encoding pair.
\end{definition}
Note that cohomological equivalence is an equivalence relation by \Cref{lem:encoding-pair-operations}. It also respects the module operations discussed in \Cref{subsec:make-pairs}.
\begin{remark}
    Fixing a $G$-module $X'$ and $r\in\ZZ$, we note that the set of $G$-modules $X$ such that $(X,X')$ is an $r$-encoding pair, if nonempty, forms exactly one equivalence class up to cohomological equivalence. Indeed, if $(A,X')$ is an $r$-encoding pair, then $(B,X')$ is an $r$-encoding pair if and only if $(A,B)$ is a $0$-encoding pair.
\end{remark}
We take a moment to give a more concrete definition of cohomological equivalence.
\begin{proposition} \label{prop:cohomologicaldef}
	Let $G$ be a finite group, and let $X$ and $X'$ be $G$-modules. Then $X$ and $X'$ are cohomologically equivalent if and only if there exist $[\varphi]\in\widehat H^0(G,\op{Hom}_\ZZ(X',X))$ and $[\varphi']\in\widehat H^0(G,\op{Hom}_\ZZ(X,X'))$ such that
	\[[\varphi\circ\varphi']=[{\id_X}]\in\widehat H^0(G,\op{Hom}_\ZZ(X,X))\qquad\text{and}\qquad[\varphi'\circ\varphi]=[{\id_{X'}}]\in\widehat H^0(G,\op{Hom}_\ZZ(X',X')).\]
\end{proposition}
\begin{proof}
    This follows from \Cref{prop:better-encoding-pair} and noting that $[\varphi\circ\varphi']=[\varphi]\cup[\varphi']$ and $[\varphi'\circ\varphi]=[\varphi']\cup[\varphi]$.
\end{proof}
Even though \Cref{prop:cohomologicaldef} suggests modules should be approximately isomorphic to be cohomologically equivalent, it can still be relatively difficult to tell if two modules are cohomologically equivalent. However, when the relevant $G$-modules are $\ZZ$-free, we can handle this somewhat. We begin by studying when a $G$-module is cohomologically equivalent to $0$.
\begin{proposition} \label{prop:bettercohomequivzero}
	Let $G$ be a finite group, and let $X$ be a $G$-module. The following are equivalent.
	\begin{enumerate}
	    \item[(a)] $X$ is cohomologically equivalent to $0$.
	    \item[(b)] $\widehat H^0(G,\op{Hom}_\ZZ(X,X))=0$.
	    \item[(c)] $X$ is induced.
	\end{enumerate}
\end{proposition}
\begin{proof}
    The equivalence of (b) and (c) follows from \cite[Proposition~2.2]{rim-group-cohomology}, so we show (a) and (b) are equivalent. Suppose $\widehat H^0(G,\op{Hom}_\ZZ(X,X))=0$. We use \Cref{prop:cohomologicaldef}: set $\varphi\colon 0\to X$ and $\varphi'\colon X\to0$ equal to the zero maps. Then note that $\op{Hom}_\ZZ(X,X)$ is induced, and $\op{Hom}_\ZZ(0,0)=0$, so
	\[\widehat H^0(G,\op{Hom}_\ZZ(X,X))=\widehat H^0(G,\op{Hom}_\ZZ(0,0))=0,\]
	making the checks on $\varphi$ and $\varphi'$ both trivial.
	
	Conversely, if $X$ is cohomologically equivalent to $0$, then there is a natural isomorphism
	\[\widehat H^0(G,\op{Hom}_\ZZ(X,-))\simeq\widehat H^0(G,\op{Hom}_\ZZ(0,-)),\]
	so it follows $\widehat H^0(G,\op{Hom}_\ZZ(X,X))=0$ by plugging in $X$.
\end{proof}
\begin{example} \label{ex:ziexample}
	It is not in general true that two $G$-modules $X$ and $X'$ are cohomologically equivalent if and only if $\widehat H^0(G,\op{Hom}_\ZZ(X,X))\cong\widehat H^0(G,\op{Hom}_\ZZ(X',X'))$. For example, we can let $G=\langle\sigma\rangle\cong\ZZ/2\ZZ$ act on $X=\ZZ$ trivially and on $X'=\ZZ i$ by $\sigma\colon i\mapsto-i$.
\end{example}
\begin{corollary} \label{cor:cohom-zero-to-zg-proj}
	Let $G$ be a finite group, and let $X$ be a $\ZZ$-free $G$-module. Then $X$ cohomologically equivalent to $0$ if and only if $X$ is $\ZZ[G]$-projective.
\end{corollary}
\begin{proof}
	This is \cite[Theorem~4.11]{rim-group-cohomology}, though we are providing a different proof. The backward direction follows from \Cref{prop:bettercohomequivzero}, so we focus on the forward direction. It suffices to show that $\op{Hom}_{\ZZ[G]}(X,-)$ is exact. Well, given a short exact sequence
	\[0\to A\to B\to C\to 0\]
	of $G$-modules, the fact that $X$ is $\ZZ$-free implies
	\[0\to\op{Hom}_\ZZ(X,A)\to\op{Hom}_\ZZ(X,B)\to\op{Hom}_\ZZ(X,C)\to0\]
	is also exact. Applying $(-)^G$ finishes because $H^1(G,\op{Hom}_\ZZ(X,A))=0$.
\end{proof}
Lastly, here is our more general $\ZZ$-free result.
\begin{proposition}
    Let $G$ be a finite group, and let $X$ and $X'$ be $\ZZ$-free $G$-modules. Then $X$ and $X'$ are cohomologically equivalent if and only if there exist $\ZZ[G]$-projective modules $P$ and $P'$ such that $X\oplus P\cong X'\oplus P'$.
\end{proposition}
\begin{proof}
    Consider the exact sequence
    \[0\to X\to\op{Hom}_\ZZ(\ZZ[G],X)\to\op{Hom}_\ZZ(I_G,X)\to0.\]
    Setting $Y\coloneqq\op{Hom}_\ZZ(I_G,X)$, we see that $(X,Y)$ is a $1$-encoding pair by \Cref{prop:encodingses}, as is $(X',Y)$. But then \Cref{thm:reason-for-encoding} implies there exist $\ZZ[G]$-projective modules $P$ and $P'$ with exact sequences
    \[\arraycolsep=1.4pt\begin{array}{ccccccccc}
        0 &\to& X &\to& P' &\to& Y &\to& 0, \\
        0 &\to& X' &\to& P &\to& Y &\to& 0.
    \end{array}\]
    It follows $X\oplus P\cong X'\oplus P'$ from \cite[Theorem~1]{kaplansky-group-cohomology}.
\end{proof}

\section{More on Encoding}

In this section, we extend the theory of encoding modules in various informative ways. For example, we will complete the proof of \Cref{thm:main}.

\subsection{Two Perfect Pairings}
A pretty way to state the theory we've built is via perfect pairings. For example, one can restate \Cref{prop:allnaturaltransarecups} as follows.
\begin{proposition}
    Let $G$ be a finite group, and let $X$ and $X'$ be $G$-modules. Then, given $r,s\in\ZZ$, the cup-product pairing induces an isomorphism of abelian groups
    \[\widehat H^r(G,\op{Hom}_\ZZ(X',X))\to\op{Mor}_2\left(\widehat H^s(G,\op{Hom}_\ZZ(X,-)),\widehat H^{r+s}(G,\op{Hom}_\ZZ(X',-))\right).\]
    Explicitly, the isomorphism is by $x\mapsto(-\cup x)$.
\end{proposition}
In general, one cannot hope to upgrade this perfect pairing from between functors to between $G$-modules because we can plug in $0$ to the functors. However, if we rearrange the items of our pairing, this becomes possible. 
\begin{lemma}
    Let $G$ be a finite group, and let $X$ be a $G$-module, and let $R\coloneqq\widehat H^0(G,\op{Hom}_\ZZ(X,X))$ be a ring with multiplication given by the cup product. For any $i\in\ZZ$, the functor $F\coloneqq\widehat H^i(G,\op{Hom}_\ZZ(X,-))$ is actually a functor from $G$-modules to right $R$-modules, where the $R$-action is given by the cup product.
\end{lemma}
\begin{proof}
    Given a $G$-module homomorphism $\varphi\colon A\to B$, we need to show that the induced morphism $F\varphi$ is a morphism of right $R$-modules. Well, we can compute that $F\varphi$ is
    \[([\varphi]\cup-)\colon\widehat H^i(G,\op{Hom}_\ZZ(X,A))\to\widehat H^i(G,\op{Hom}_\ZZ(X,B)),\]
    which does commute with the $R$-action on the right.
\end{proof}
\begin{lemma} \label{lem:upgrade-natural-trans}
    Let $G$ be a finite group, and let $(X,X')$ be an $r$-encoding pair. Set $R\coloneqq\widehat H^0(G,\op{Hom}_\ZZ(X,X))$ and $F_i\coloneqq\widehat H^i(G,\op{Hom}_\ZZ(X,-))$, where $i\in\ZZ$ is arbitrary, and define $R'$ and $F'_i$ similarly. Given an encoding element $x$, the natural isomorphisms $(-\cup x)\colon F_i\simeq F'_{i+r}$ induce an isomorphism
    \[\op{Hom}_R\left(F_iA,F_jB\right)\cong\op{Hom}_{R'}\left(F'_{i+r}A,F'_{j+r}B\right)\]
    for any $G$-modules $A,B$ and $i,j\in\ZZ$.
\end{lemma}
\begin{proof}
    Let $x'$ be a decoding element. The main point is to draw the diagram
    \[\begin{tikzcd}
    	{\widehat H^i(G,\op{Hom}_\ZZ(X,A))} & {\widehat H^j(G,\op{Hom}_\ZZ(X,B))} \\
    	{\widehat H^{i+r}(G,\op{Hom}_\ZZ(X',A))} & {\widehat H^{j+r}(G,\op{Hom}_\ZZ(X',B))}
    	\arrow[from=1-1, to=1-2]
    	\arrow[from=2-1, to=2-2]
    	\arrow["{(-\cup x)}"', shift right=2, from=1-1, to=2-1]
    	\arrow["{(-\cup x')}"', shift right=2, from=2-1, to=1-1]
    	\arrow["{(-\cup x)}"', shift right=2, from=1-2, to=2-2]
    	\arrow["{(-\cup x')}"', shift right=2, from=2-2, to=1-2]
    \end{tikzcd}\]
    which explains how to turn morphisms $F_iA\to F'_jB$ to morphisms $F_{i+r}A\to F'_{j+r}B$ (and vice versa). One can check that $R$-module morphisms go to $R'$-module morphisms (and vice versa), which finishes.
\end{proof}
\begin{theorem} \label{thm:abstractperfectpairing}
	Let $G$ be a finite group, and let $(X,X')$ be an $r$-encoding pair. Set $R'\coloneqq\widehat H^0(G,\op{Hom}_\ZZ(X',X'))$. For any $G$-module $A$ and $i\in\ZZ$, the cup-product pairing induces an isomorphism
	\[\widehat H^i(G,\op{Hom}_\ZZ(X,A))\to\op{Hom}_{R'}\left(\widehat H^{r}(G,\op{Hom}_\ZZ(X',X)),\widehat H^{i+r}(G,\op{Hom}_\ZZ(X',A))\right)\]
	of abelian groups.
\end{theorem}
\begin{proof}
    Set $R\coloneqq\widehat H^0(G,\op{Hom}_\ZZ(X,X))$. Tracking through \Cref{lem:upgrade-natural-trans}, it suffices to show the cup-product pairing induces an isomorphism
    \[\widehat H^i(G,\op{Hom}_\ZZ(X,A)) \to \op{Hom}_R\left(\widehat H^0(G,\op{Hom}_\ZZ(X,X)),\widehat H^i(G,\op{Hom}_\ZZ(X,A))\right).\]
	This is a consequence of how the $R$-action is defined.
\end{proof}
\begin{remark}
    In the case of $r$-encoding modules $X$ so that $(\ZZ,X)$ is a $(-r)$-encoding pair, the fact that $\widehat H^0(G,\op{Hom}_\ZZ(X,X))$ is isomorphic to the ring $\ZZ/|G|\ZZ$ (discussed in \Cref{rem:encoding-module-ring}) makes \Cref{thm:abstractperfectpairing} read
    \[\widehat H^i(G,A)\cong\op{Hom}_\ZZ\left(\widehat H^{-r}(G,\op{Hom}_\ZZ(X,\ZZ)),\widehat H^{i-r}(G,\op{Hom}_\ZZ(X,A))\right).\]
\end{remark}

\subsection{Encoding by Tensoring}
It turns out that we can also encode ``on the other side,'' in the following sense.
\begin{proposition} \label{prop:encode-to-tensor}
	Let $G$ be a finite group, and let $(X,X')$ be an $r$-encoding pair with encoding element $x\in\widehat H^r(G,\op{Hom}_\ZZ(X',X))$. Then the cup products
	\[(x\cup-)\colon\widehat H^i(G,X'\otimes_\ZZ-)\Rightarrow\widehat H^{i+r}(G,X\otimes_\ZZ-)\]
	assemble into a natural isomorphism for any $i\in\ZZ$. As usual, the cup product is induced by evaluation.
\end{proposition}
\begin{proof}
    Pick up a decoding element $x'\in\widehat H^{i+r}(G,\op{Hom}_\ZZ(X,X'))$ so that
	\[x\cup x'=[\id_{X}]\qquad\text{and}\qquad x'\cup x=[{\id_{X'}}].\]
	Thus, as in the proof of \Cref{prop:better-encoding-pair}, we can compute that
	\[\arraycolsep=1.4pt\begin{array}{rccc}
		(x\cup-)\colon& \widehat H^i(G,X'\otimes_\ZZ-) &\to& \widehat H^{i+r}(G,X\otimes_\ZZ-) \\
		(x'\cup-)\colon& \widehat H^{i+r}(G,X\otimes_\ZZ-) &\to& \widehat H^i(G,X'\otimes_\ZZ-)
	\end{array}\]
	are inverse natural transformations, which finishes the proof.
\end{proof}
Amusingly, we can plug in $\ZZ$ into \Cref{prop:encode-to-tensor} to compute some cohomology groups.
\begin{corollary} \label{cor:compute-cohom-pairs}
	Let $G$ be a finite group, and let $(X,X')$ be an $r$-encoding pair with encoding element $x\in\widehat H^r(G,\op{Hom}_\ZZ(X',X))$. Then the cup products
	\[(x\cup-)\colon\widehat H^i(G,X')\to\widehat H^{i+r}(G,X)\]
	are isomorphisms for all $i\in\ZZ$.
\end{corollary}
With some care, we can even build a converse for \Cref{prop:encode-to-tensor}.
\begin{proposition} \label{prop:tensor-to-encode}
    Let $G$ be a finite group, and let $X$ and $X'$ be finitely generated $G$-modules equipped with a natural isomorphism
    \[\widehat H^i(G,X'\otimes_\ZZ-)\simeq\widehat H^{i+r}(G,X\otimes_\ZZ-)\]
    for some $i,r\in\ZZ$. Then $(X,X')$ is an $r$-encoding pair.
\end{proposition}
\begin{proof}
    The idea is to work with finitely generated modules and use duality. Indeed, for a finitely generated $\ZZ[G]$-module $A$, we set $A^*\coloneqq\op{Hom}_\ZZ(A,\QQ/\ZZ)$ and compute
    \begin{align*}
        \widehat H^{-i-1}(G,\op{Hom}_\ZZ(X,A)) &\cong \widehat H^{-i-1}(G,\op{Hom}_\ZZ(X,\op{Hom}_\ZZ(A^*,\QQ/\ZZ))) \\
        &\cong \widehat H^{-i-1}(G,\op{Hom}_\ZZ(X\otimes_\ZZ A^*,\QQ/\ZZ)) \\
        &\cong \op{Hom}_\ZZ\left(\widehat H^i(G,X\otimes_\ZZ A^*),\widehat H^{-1}(G,\QQ/\ZZ)\right).
    \end{align*}
    where in the last isomorphism we have used the duality theorem \cite[Corollary~XII.6.5]{cartan-eilenberg}. We can now apply the hypothesized natural isomorphism and work backwards to assemble a natural isomorphism
    \[\widehat H^{-i-1}(G,\op{Hom}_\ZZ(X,-))\simeq\widehat H^{-i-1+r}(G,\op{Hom}_\ZZ(X',-)),\]
    where the functors are $\mathrm{FgMod}_G\to\mathrm{Ab}$. (Here, $\mathrm{FgMod}_G$ is the category of finitely generated $G$-modules.) Now, we would be done if we could upgrade this natural isomorphism to $\mathrm{Mod}_G\to\mathrm{Ab}$.
    
    Well, because $X$ and $X'$ are finitely generated, the arguments of \Cref{prop:allnaturaltransarecups} and \Cref{prop:better-encoding-pair} generalize from functors $\mathrm{Mod}_G\to\mathrm{Ab}$ to functors $\mathrm{FgMod}_G\to\mathrm{Ab}$. This gives $x\in\widehat H^r(G,\op{Hom}_\ZZ(X',X))$ and $x'\in\widehat H^{-r}(G,\op{Hom}_\ZZ(X,X'))$ such that $x'\cup x=[{\id}_{X'}]$ and $x\cup x'=[{\id_X}]$, so $(X,X')$ is an $r$-encoding pair.
\end{proof}
\begin{remark}
    It is conceivable that one can drop the finitely generated hypothesis in \Cref{prop:tensor-to-encode}. This seems to be difficult because the analogous version of \Cref{prop:allnaturaltransarecups} is nontrivial to prove. Namely, if we hypothesize that our natural transformation and its inverse arise from cup products, then the converse holds.
\end{remark}

\subsection{Torsion-Free Encoding}
For this subsection, we will focus on encoding modules instead of encoding pairs, for technical reasons. In the theory of periodic cohomology, one can show that it is enough to check the single cohomology group
\[\widehat H^r(G,\ZZ)\cong\ZZ/|G|\ZZ\]
for some $r\in\ZZ$ to get $r$-periodic cohomology. We might hope that something similar is true for our $r$-encoding modules. A main ingredient in the proof of this statement for periodic cohomology is an integral duality theorem, so here is the needed analogue.
\begin{proposition}[{\cite[Exercise~VI.7.3]{brown-cohomology}}] \label{prop:abstractintegralduality}
	Let $G$ be a finite group, and let $X$ be a $\ZZ$-free $G$-module. Then the cup-product pairing induces an isomorphism
	\[\widehat H^i(G,\op{Hom}_\ZZ(X,\ZZ))\to\op{Hom}_\ZZ\left(\widehat H^{-i}(G,X),\widehat H^0(G,\ZZ)\right)\]
	by $x\mapsto(x\cup-)$, for all $i\in\ZZ$. Indeed, this is a duality upon identifying $\widehat H^0(G,\ZZ)$ with $\frac1{|G|}\ZZ/\ZZ\subseteq\QQ/\ZZ$.
\end{proposition}
\begin{proof}
	This proof is analogous to \cite[Theorem~XII.6.6]{cartan-eilenberg}. The key to the proof is the short exact sequence
	\[0\to\ZZ\to\QQ\to\QQ/\ZZ\to0.\]
	Because $X$ is $\ZZ$-free,
	\[0\to\op{Hom}_\ZZ(X,\ZZ)\to\op{Hom}_\ZZ(X,\QQ)\to\op{Hom}_\ZZ(X,\QQ/\ZZ)\to0\]
	is also short exact. Quickly, note that $\QQ$ and $\op{Hom}_\ZZ(X,\QQ)$ are both divisible abelian groups and therefore acyclic, so the induced boundary morphisms will be isomorphisms. As such, we compute
	\begin{align*}
	    \widehat H^i(G,\op{Hom}_\ZZ(X,\ZZ)) &\cong \widehat H^{i-1}(G,\op{Hom}_\ZZ(X,\QQ/\ZZ)) \\
	    &\cong \op{Hom}_\ZZ\left(\widehat H^{-i}(G,X),\widehat H^{-1}(G,\QQ/\ZZ)\right) \\
	    &\cong \op{Hom}_\ZZ\left(\widehat H^{-i}(G,X),\widehat H^0(G,\ZZ)\right),
	\end{align*}
	where we have used the duality theorem \cite[Corollary~XII.6.5]{cartan-eilenberg} at the second isomorphism. Tracking these isomorphisms through reveals that the composite is induced by the cup-product pairing.
\end{proof}
And here is our result.
\begin{proposition} \label{prop:finitecohomcheck}
	Let $G$ be a finite group, and let $X$ be a $\ZZ$-free $G$-module. The following are equivalent.
	\begin{enumerate}
		\item[(a)] $X$ is an $r$-encoding module.
		\item[(b)] $\widehat H^r(G,X)\cong\ZZ/|G|\ZZ$ and $\widehat H^0(G,\op{Hom}_\ZZ(X,X))\cong\ZZ/|G|\ZZ$.
		\item[(c)] $\widehat H^r(G,X)\cong\ZZ/|G|\ZZ$ and $\widehat H^0(G,\op{Hom}_\ZZ(X,X))$ is cyclic.
	\end{enumerate}
\end{proposition}
\begin{proof}
	For brevity, set $n\coloneqq|G|$. That (a) implies (b) is not hard: \Cref{cor:compute-cohom-pairs} tells us that $\widehat H^r(G,X)\cong\ZZ/n\ZZ$, and then being an $r$-encoding module promises an isomorphism
	\[\widehat H^0(G,\op{Hom}_\ZZ(X,X))\cong\widehat H^r(G,X)\cong\ZZ/n\ZZ.\]
	Continuing, we see that (b) implies (c) easily, so we have left to show that (c) implies (a).
	
	For this, we use \Cref{prop:abstractintegralduality,prop:better-encoding-pair}. We are given $x\in\widehat H^r(G,X)$ of order $n$, so we note that there is a morphism
	\[\widehat H^r(G,X)\cong\ZZ/n\ZZ=\widehat H^0(G,\ZZ)\]
	sending $x$ to $[1]$. Thus, \Cref{prop:abstractintegralduality} grants $x'\in\widehat H^{-r}(G,\op{Hom}_\ZZ(X,\ZZ))$ such that $x'\cup x=[1]$.
	It remains to check $x\cup x'=[{\id_X}]\in\widehat H^0(G,\op{Hom}_\ZZ(X,X))$. This is more difficult. We claim
	\[(-\cup x)\colon\widehat H^0(G,\op{Hom}_\ZZ(X,X))\to\widehat H^r(G,X)\]
	is injective. This will be enough because
	\[(x\cup x')\cup x=x\cup(x'\cup x)=x\cup[1]=x=[{\id_X}]\cup x.\]
	To see the claim, note $[{\id_X}]\in\widehat H^0(G,\op{Hom}_\ZZ(X,X))$ has order divisible by $n$: if $k[{\id_X}]=0$, then $0=k(x\cup[{\id_X}])=kx$, so $n\mid k$. But $\widehat H^0(G,\op{Hom}_\ZZ(X,X))$ is cyclic and $n$-torsion, so $\widehat H^0(G,\op{Hom}_\ZZ(X,X))$ is cyclic of order $n$ generated by $[{\id_X}]$. Thus, we note that there is a unique morphism
	\[\widehat H^0(G,\op{Hom}_\ZZ(X,X))\cong\ZZ/n\ZZ\cong\widehat H^r(G,X)\]
	sending $[{\id_X}]$ to $1$ to $x$, and this map is an isomorphism. However, $(-\cup x)$ sends $[{\id_X}]$ to $x$ as well, so $(-\cup x)$ is an isomorphism and thus injective.
\end{proof}
    
\begin{example}
	The $\ZZ$-free condition is necessary. Let $G\coloneqq\ZZ/2\ZZ$ act on $X\coloneqq\ZZ/2\ZZ$ trivially so that
	\[\widehat H^0(G,\op{Hom}_\ZZ(X,X))\cong\widehat H^0(G,X)\cong\widehat H^{-1}(G,X)=\ZZ/2\ZZ.\]
	However, $X$ is not a $(-1)$-encoding module because
	\[\widehat H^1(G,\op{Hom}_\ZZ(X,\ZZ))\cong\widehat H^1(G,0)=0\not\cong\ZZ/2\ZZ=\widehat H^0(G,\ZZ).\]
\end{example}
With such a simple condition for $r$-encoding, we can begin to imagine stronger classification results. For example, here is a version of indecomposability.
\begin{corollary} \label{cor:indecomposable}
	Let $G$ be a finite $p$-group. If $A\oplus B$ is a finitely generated $\ZZ$-free $r$-encoding module, then one of $A$ or $B$ is an $r$-encoding module, and the other is cohomologically equivalent to $0$.
\end{corollary}
\begin{proof}
	We use \Cref{prop:finitecohomcheck} repeatedly. On one hand, we have an embedding
	\begin{equation}
	    \widehat H^0(G,\op{Hom}_\ZZ(A,A))\oplus\widehat H^0(G,\op{Hom}_\ZZ(B,B))\into\widehat H^0(G,\op{Hom}_\ZZ(A\oplus B,A\oplus B)), \label{eq:cohom-summands}
	\end{equation}
	so both $\widehat H^0(G,\op{Hom}_\ZZ(A,A))$ and $\widehat H^0(G,\op{Hom}_\ZZ(B,B))$ are cyclic because $\widehat H^0(G,\op{Hom}_\ZZ(A\oplus B,A\oplus B))$ is. On the other hand, we note
	\[\widehat H^r(G,A)\oplus\widehat H^r(G,B)\cong\widehat H^r(G,A\oplus B)\cong\ZZ/|G|\ZZ\]
	forces $\widehat H^r(G,A)\cong\ZZ/|G|\ZZ$ or $\widehat H^r(G,B)\cong\ZZ/|G|\ZZ$ because $G$ is a finite $p$-group.
	
	Thus, one of $A$ or $B$ is an $r$-encoding module. Without loss of generality, say that $A$ is. It remains to show that $B$ is cohomologically equivalent to $0$. Well, $\widehat H^0(G,\op{Hom}_\ZZ(A,A))\cong\ZZ/|G|\ZZ$, so the embedding \eqref{eq:cohom-summands} forces $\widehat H^0(G,\op{Hom}_\ZZ(B,B))=0$, which finishes by \Cref{prop:bettercohomequivzero}.
\end{proof}

\section{Acknowledgements}
This research was conducted at the University of Michigan REU during the summer of 2022. It was supported by NSF DMS-1840234. The author would especially like to thank his advisors Alexander Bertoloni Meli, Patrick Daniels, and Peter Dillery for their eternal patience and guidance, in addition to a number of helpful comments on earlier drafts of this article. Without their advice, this project would have been impossible. The author would also like to thank Maxwell Ye for a number of helpful conversations and consistent companionship. Without him, the author would have been left floating adrift and soulless. Additionally, the author would like to thank the anonymous referee for a careful reading and many detailed comments on an earlier draft.

\bibliographystyle{plain}

\end{document}